\documentclass[12pt]{amsart}
\usepackage{amssymb,amsthm,amsmath,epsfig,latexsym}
\usepackage{calc,times}
\begin{document}

\newcommand{\mmbox}[1]{\mbox{${#1}$}}
\newcommand{\proj}[1]{\mmbox{{\mathbb P}^{#1}}}
\newcommand{\Cr}{C^r(\Delta)}
\newcommand{\CR}{C^r(\hat\Delta)}
\newcommand{\affine}[1]{\mmbox{{\mathbb A}^{#1}}}
\newcommand{\Ann}[1]{\mmbox{{\rm Ann}({#1})}}
\newcommand{\caps}[3]{\mmbox{{#1}_{#2} \cap \ldots \cap {#1}_{#3}}}
\newcommand{\N}{{\mathbb N}}
\newcommand{\Z}{{\mathbb Z}}
\newcommand{\R}{{\mathbb R}}
\newcommand{\Tor}{\mathop{\rm Tor}\nolimits}
\newcommand{\Ext}{\mathop{\rm Ext}\nolimits}
\newcommand{\Hom}{\mathop{\rm Hom}\nolimits}
\newcommand{\im}{\mathop{\rm Im}\nolimits}
\newcommand{\rank}{\mathop{\rm rank}\nolimits}
\newcommand{\supp}{\mathop{\rm supp}\nolimits}
\newcommand{\arrow}[1]{\stackrel{#1}{\longrightarrow}}
\newcommand{\CB}{Cayley-Bacharach}
\newcommand{\coker}{\mathop{\rm coker}\nolimits}
\sloppy
\newtheorem{defn0}{Definition}[section]
\newtheorem{prop0}[defn0]{Proposition}
\newtheorem{quest0}[defn0]{Question}
\newtheorem{thm0}[defn0]{Theorem}
\newtheorem{lem0}[defn0]{Lemma}
\newtheorem{corollary0}[defn0]{Corollary}
\newtheorem{example0}[defn0]{Example}
\newtheorem{remark0}[defn0]{Remark}

\newenvironment{defn}{\begin{defn0}}{\end{defn0}}
\newenvironment{prop}{\begin{prop0}}{\end{prop0}}
\newenvironment{quest}{\begin{quest0}}{\end{quest0}}
\newenvironment{thm}{\begin{thm0}}{\end{thm0}}
\newenvironment{lem}{\begin{lem0}}{\end{lem0}}
\newenvironment{cor}{\begin{corollary0}}{\end{corollary0}}
\newenvironment{exm}{\begin{example0}\rm}{\end{example0}}
\newenvironment{rem}{\begin{remark0}\rm}{\end{remark0}}

\newcommand{\defref}[1]{Definition~\ref{#1}}
\newcommand{\propref}[1]{Proposition~\ref{#1}}
\newcommand{\thmref}[1]{Theorem~\ref{#1}}
\newcommand{\lemref}[1]{Lemma~\ref{#1}}
\newcommand{\corref}[1]{Corollary~\ref{#1}}
\newcommand{\exref}[1]{Example~\ref{#1}}
\newcommand{\secref}[1]{Section~\ref{#1}}
\newcommand{\remref}[1]{Remark~\ref{#1}}
\newcommand{\questref}[1]{Question~\ref{#1}}

\newcommand{\std}{Gr\"{o}bner}
\newcommand{\jq}{J_{Q}}


\title{Finding Inverse Systems from Coordinates}
\author{Stefan O. Tohaneanu}
\thanks{Author's address: Department of Mathematics, University of Idaho, Moscow, Idaho 83844, USA; Email: tohaneanu@uidaho.edu}


\begin{abstract}
Let $I$ be a homogeneous ideal in $R=\mathbb K[x_0,\ldots,x_n]$, such that $R/I$ is an Artinian Gorenstein ring. A famous theorem of Macaulay says that in this instance $I$ is the ideal of polynomial differential operators with constant coefficients that cancel the same homogeneous polynomial $F$. A major question related to this result is to be able to describe $F$ in terms of the ideal $I$. In this note we give a partial answer to this question, by analyzing the case when $I$ is the Artinian reduction of the ideal of a reduced (arithmetically) Gorenstein zero-dimensional scheme $\Gamma\subset\mathbb P^n$. We obtain $F$ from the coordinates of the points of $\Gamma$. \end{abstract}
\maketitle

\section{Introduction}

Let $\mathbb K$ be a field of characteristic zero, and let $R=\mathbb K[x_0,\ldots,x_n]$ be the ring of homogeneous polynomials with coefficients in $\mathbb K$. Let $I\subset R$ be a homogeneous ideal.

The ring $R/I$ is called {\em Artinian} if $R/I$ is a finite dimensional vector space over $\mathbb K$. Equivalently, there exists a positive integer $d>0$ such that $(R/I)_d=0$ (i.e., every homogeneous polynomial of degree $d$ is an element of $I$).

An Artinian ring $R/I$ is called {\em Gorenstein} if the {\em socle} $$Soc(R/I):=\{a\in R/I \,|\, \langle x_0,\ldots,x_n\rangle a=0\}$$ is a 1-dimensional $\mathbb K-$vector space. If $s+1$ is the least integer such that $(R/I)_{s+1}=0$ and if $R/I$ is Gorenstein, then $Soc(R/I)=(R/I)_s$ and therefore, ${\rm dim}(R/I)_s=1$. In this instance $s$ is called the {\em socle degree} of $R/I$.

An {\em arithmetically Gorenstein} scheme means a projective scheme whose coordinate ring localized at any of its prime ideals is a local Gorenstein ring (i.e., it is a local ring that is Cohen-Macaulay and the canonical module is free of rank 1). In terms of the graded minimal free resolution, if $X\subset\mathbb P^n$ is a $d-$dimensional scheme with defining ideal $I_X$, then $X$ is arithmetically Gorenstein if and only if $R/I_X$ has the graded minimal free resolution as an $R-$module: $$0\rightarrow F_k\rightarrow\cdots\rightarrow F_1\rightarrow R\rightarrow R/I_X\rightarrow 0,$$ where $k=n-d$ and $F_k\simeq R(-\alpha)$. $\alpha-k$ will be called also the {\em socle degree} of $X$, for obvious reasons, and coincides with the Castelnuovo-Mumford regularity $reg(R/I_X)$ of $X$.

\medskip

A famous theorem of Macaulay (\cite{Ma}), known as {\em Macaulay's Inverse System Theorem}, states that the Artinian ring $R/I$ is Gorenstein if and only if $I$ is the ideal of a system of homogeneous polynomial differential operators with constant coefficients having a unique solution. More precisely, let $S=\mathbb K[y_0,\ldots,y_n]$ be the homogeneous polynomial ring with coefficients in $\mathbb K$ and variables $y_0,\ldots,y_n$. $R$ acts on $S$ by $$x_0^{i_0}\cdots x_n^{i_n}\circ y_0^{j_0}\cdots y_n^{j_n}= \frac{\partial^{i_0+\cdots+i_n}}{\partial y_0^{i_0}\cdots \partial y_n^{i_n}}(y_0^{j_0}\cdots y_n^{j_n}),$$ extended by linearity. Then $R/I$ is Artinian Gorenstein if and only if $I=Ann(F):=\{f\in R | f\circ F=0\}$, for some $F\in S$ (we are going to denote the elements in $S$ with capital letters). The best surveys on applications of inverse systems and also very good introductions to this subject are \cite{Ge} and \cite{IaKa}, and the citations therein.

The polynomial $F$ is roughly what is known as the inverse system of the Artinian Gorenstein ideal $I$. In general, for any ideal $I$, {\em the inverse system} of $I$, is by definition: $I^{-1}:=\oplus_j (I^{-1})_j$, where $$(I^{-1})_j:=\{G\in S_j|f\circ G=0,\mbox{ for any }f\in I_j\}.$$ If $I=Ann(F)$, then $(I^{-1})_j$ is the $\mathbb K-$vector space spanned by the partial derivatives of order $\deg(F)-j$ of $F$.

One of the questions in the field is to determine $F$ from the ideal $I$. It is not known how the shape of $F$ makes the distinction between Artinian complete intersections and Artinian Gorenstein rings, as the first class is included in the second. On this idea, a question asked by Tony Geramita is if one can determine $F$ from the minimal generators of an Artinian complete intersection ideal $I$, at the same time making this distinction. More generally, one would want to determine $I^{-1}$ from $I$ and conversely. Cho and Iarrobino (\cite{ChIa}) made some progress on this direction when $I$ is the defining ideal of a zero-dimensional scheme in $\mathbb P^n$, saturated (\cite[Proposition 1.13]{ChIa}) or locally Gorenstein (\cite[Theorem 3.3]{ChIa}). These results concern finding $I^{-1}$ from the generators of $I$ and conversely, finding the generators of $I$ from $I^{-1}$.

Our notes follow the same direction: we determine $F$ for the case when $I$ is the Artinian reduction\footnote{By Artinian reduction we understand Artinian reduction by a general linear form.} of the ideal of a zero-dimensional reduced Gorenstein scheme (i.e., a Gorenstein set of points). Our Theorem \ref{thm:main_theorem} shows that $F$ is determined uniquely from the homogeneous coordinates of the points that form this scheme: $F=\sum c_i L_i^r$, where $L_i$ is the dual form of each point $P_i$ in this set, $r$ is the regularity, and $c_i$ are nonzero constants, the sum being taken over all points in the set. This result resembles to \cite[Theorem 3.8]{BrCoMoTs}, the difference being that for a Gorenstein set of points, we specify who are the constants $c_i$ in the decomposition $F=\sum c_i L_i^r$, and our approach is more homological than computational. \footnote{ Our main result in a way determines the inverse system from $V(I)$, rather than from $I$. A converse to this approach means to solve systems of multivariate polynomials that have a zero-dimensional saturated solution $V(I)$, by using inverse systems $I^{-1}$ (see \cite{MoPa} for a detailed analysis).}

Inverse systems occur naturally in the theory of systems of PDE's with constant coefficients, and similar results to the ones obtained via commutative algebraic methods appeared in the literature from this direction of study (see \cite{Re}). \cite{St} determines the dimension of $I^{-1}$, thought also as the solution of such systems of PDE's, in the generic case and when $R/I$ is Artinian. Also, inverse systems are put to great use in the theory of splines approximation (e.g., \cite{GeSc}), and also in the study of Weak Lefschetz Property of Artinian algebras (e.g., \cite{MiMiNa2}, \cite{HaSeSc}). Concerning the later, it would be interesting to see if our main result can bring some insights towards answering \cite[Question 3.8]{MiNa2}. Also, one should mention the application of inverse systems in `The Waring's Problem'' and tensor decompositions (see \cite{Ge}, \cite{IaKa}, or recent progress in \cite{Br}, or \cite{BrCoMoTs}).

\section{Inverse systems of Artinian reductions of reduced zero-dimensional Gorenstein schemes}

In this section we present and prove the main theorem. The proof makes use of two classical results in the theory of inverse systems and Gorenstein algebras.

\medskip

$\bullet$ \textbf{Emsalem-Iarrobino Theorem} (\cite[Theorems IIA and IIB]{EnIa}). In our situation of reduced zero-dimensional schemes this theorem says the following (\cite[Theorem I]{EnIa}): suppose $Z=\{P_1,\ldots,P_m\}\subset \mathbb P^n$, and let $L_i\in S$ be the associated (dual) linear form of $P_i$. Then $(I_Z^{-1})_j={\rm Span}_{\mathbb K}\langle L_1^j,\ldots,L_m^j\rangle$, and the Hilbert function satisfies $HF(R/I_Z,j)=\dim_{\mathbb K}(I_Z^{-1})_j$.

\medskip

$\bullet$ \textbf{Davis-Geramita-Orecchia Theorem} (\cite[Theorem 5]{DaGeOr}). As stated in \cite[Theorem 1]{MiNa}, the theorem is the following: a reduced set of points $Z$ is arithmetically Gorenstein if and only if its $h-$vector is symmetric and it has the Cayley-Bacharach property.

The {\em $h-$vector} of $Z$ is the vector of the Hilbert function values of the Artinian reduction of $I_Z$, the ideal of $Z$. One says that $Z$ has {\em the Cayley-Bacharach property} if for every subset $Y\subset Z$ of cardinality $|Z|-1$, one has $HF(R/I_Y,s-1)=HF(R/I_Z,s-1)$, where $s$ is the last degree where the $h-$vector of $Z$ is non-zero.

\medskip

We also use a corollary to the {\bf Kustin-Ulrich Socle Lemma} (\cite[Lemma 1.1]{KuUl}). For an Artinian ring $R/I$ as seen in the introduction, denote $s(R/I)$ the minimum positive integer $d$ such that $(R/I)_{d+1}=0$ and $(R/I)_d\neq 0$. In case $R/I$ is Gorenstein, $s(R/I)$ is the socle degree of $R/I$. The corollary is the following:

\begin{lem}\label{lem:simple_lemma} Let $I$ and $J$ be two ideals in $R$, such that $I\subseteq J$, and $R/I$ and $R/J$ are both Artinian Gorenstein rings. If $s(R/I)=s(R/J)$, then $I=J$.
\end{lem}
\begin{proof} The proof follows immediately from the cited lemma. Just observe that in the case of Gorenstein Artinian rings, $Soc(R/I)$ and $Soc(R/J)$ are one-dimensional graded vector spaces generated in degrees $s(R/I)$ and $s(R/J)$, respectively. And these degrees are equal from hypotheses.
\end{proof}

\bigskip

Let $Z=\{P_1,\ldots,P_m\}\subset \mathbb P^n$ be a reduced zero-dimensional scheme. Let $I_Z\subset R:=\mathbb K[x_0,\ldots,x_n]$ be the ideal of $Z$, and assume $reg(R/I_Z)=r$. Suppose that $P_i=[a_{i,0},a_{i,1},\ldots,a_{i,n}]$, for $i=1,\ldots,m$, and for each $P_i$ consider the associated (dual) linear form $$L_i=a_{i,0}y_0+a_{i,1}y_1+\cdots+a_{i,n}y_n\in S:= \mathbb K[y_0,\ldots,y_n].$$

\begin{thm}\label{thm:main_theorem} If $Z$ is arithmetically Gorenstein and if $\ell\in R$ is a linear form such that $\ell(P_i)\neq 0$ for all $i=1,\ldots,m$ (i.e., $\ell$ is a non-zero divisor in $R/I_Z$), then $$\langle I_Z,\ell\rangle = Ann(c_1L_1^r+\cdots+c_mL_m^r),$$ where $c_i$'s are the unique (up to multiplication by a non-zero scalar) non-zero constants such that $$c_1\ell(P_1)L_1^{r-1}+\cdots+c_m\ell(P_m)L_m^{r-1}=0.$$

Conversely, let $F=c_1L_1^r+\cdots+c_mL_m^r\in S, c_i\neq 0$ be a form of degree $r$ such that
\begin{enumerate}
  \item $\dim_{\mathbb K}{\rm Span}_{\mathbb K}\langle L_1^r,\ldots,L_m^r\rangle=m$, and
  \item there exist the non-zero constants $d_1,\ldots,d_m$, unique up to multiplication by a non-zero scalar such that $d_1L_1^{r-1}+\cdots+d_mL_m^{r-1}=0$.

\end{enumerate} Let $Z\subset \mathbb P^n$ be the set of points dual to the linear forms $L_i$. Assuming that $Z$ has symmetric $h-$vector, then $Z$ is arithmetically Gorenstein of regularity $r$. Furthermore, $Ann(F)=\langle I_Z,\tilde{\ell}\rangle$, for some linear form $\tilde{\ell}\in R$, if and only if $\tilde{\ell}(P_i)=\frac{d_i}{c_i}$, for all $P_i\in Z$.
\end{thm}
\begin{proof} For the direct implication, if we denote with $F=c_1L_1^r+\cdots+c_mL_m^r$, it is enough to show $\langle I_Z,\ell\rangle \subseteq Ann(F)$. If we know this inclusion, then the equality follows immediately from Lemma \ref{lem:simple_lemma}.

\medskip

Let $f\in I_Z$ of degree $d$. Then $f\circ L_j^r=0$ for all $j=1,\ldots,m$. This is obvious if $d\geq r+1$. Otherwise, suppose $d\leq r$, and suppose $f=\sum_{i_0+\cdots+i_n=d} \alpha_{i_0,\ldots,i_n} x_0^{i_0}\cdots x_n^{i_n}.$

Observe that $x_0^{i_0}\cdots x_n^{i_n}\circ L_j^r=x_1^{i_1}\cdots x_n^{i_n}\circ \frac{r!}{(r-i_0)!}a_{j,0}^{i_0}L_j^{r-i_0}= \cdots=\frac{r!}{(r-d)!}a_{j,0}^{i_0}\cdots a_{j,n}^{i_n}L_j^{r-d}.$ So $$f\circ L_j^r=\frac{r!}{(r-d)!}f(P_j)L_j^{r-d}=0.$$ We obtain $$I_Z\subset Ann(c_1L_1^r+\cdots+c_mL_m^r),\mbox{ for any constants }c_i\in\mathbb K.$$ To mention here that this argument (often referred to as a ``folk-theorem''; e.g. \cite[Page 183]{Re}) is at the base of all the important results in the theory of inverse systems, as can be observed, for example, in \cite{Ge} and \cite{IaKa}.

\medskip

Now we need to show that there exist the nonzero constants $c_1,\ldots,c_m\in \mathbb K$, unique up to multiplication by a non-zero scalar, such that $$\ell \circ (c_1L_1^r+\cdots+c_mL_m^r) =0.$$ Equivalently, we have to find the unique, non-zero $d_i$ such that $$\underbrace{c_1\ell(P_1)}_{d_1}L_1^{r-1}+\cdots+\underbrace{c_m\ell(P_m)}_{d_m}L_m^{r-1}=0.$$

From Emsalem-Iarrobino Theorem one has $$HF(R/I_Z,j)=\dim_{\mathbb K}{\rm Span}_{\mathbb K}\langle L_1^j,\ldots,L_m^j\rangle.$$

$\bullet$ Since $r=reg(R/I_Z)$, then $HF(R/I_Z,r)=\deg(Z)=m$, giving $$\dim_{\mathbb K}{\rm Span}_{\mathbb K}\langle L_1^r,\ldots,L_m^r\rangle=m.$$

$\bullet$ Since $R/\langle I_Z,\ell\rangle$ is Artinian Gorenstein, and $\ell$ is a nonzero divisor in $R/I_Z$, then $HF(R/I_Z,r-1)=HF(R/I_Z,r)-HF(R/\langle I_Z,\ell\rangle,r)=m-1$, giving $$\dim_{\mathbb K}{\rm Span}_{\mathbb K}\langle L_1^{r-1},\ldots,L_m^{r-1}\rangle=m-1.$$

$\bullet$ Since $Z$ is arithmetically Gorenstein (reduced) finite set of points, then, from Davis-Geramita-Orecchia Theorem, $Z$ has the Cayley-Bacharach property, which via Emsalem-Iarrobino Theorem, translates into: for any $1\leq i_1<\cdots<i_{m-1}\leq m$, $L_{i_1}^{r-1},\ldots,L_{i_{m-1}}^{r-1}$ are linearly independent.

The second bullet gives the existence of the constants $d_i$ such that $d_1L_1^{r-1}+\cdots+d_mL_m^{r-1}=0$, whereas the third bullet proves that all $d_i$'s are non-zero and unique, up to multiplication by a non-zero scalar.

\vskip .2in

For the converse statement, first observe that by using Emsalem-Iarrobino Theorem , we have $HF(R/I_Z,r)=\deg(Z)=m$ from the assumption (1), and $HF(R/I_Z,r-1)\leq m-1$, from the assumption (2). So $reg(R/I_Z)=r$.

The assumption (2) also implies that for any $1\leq i_1<\cdots<i_{m-1}\leq m$, $L_{i_1}^{r-1},\ldots, L_{i_{m-1}}^{r-1}$ are linearly independent. So $Z$ has the Cayley-Bacharach property.

From Davis-Geramita-Orecchia Theorem with the assumption that $Z$ has symmetric $h-$vector one obtains that $Z$ is arithmetically Gorenstein.

The last statement is an immediate consequence of the direct implication we showed in the first part of this proof.
\end{proof}

The next two remarks put in a different perspective the conditions in the converse statement in the above theorem.

\begin{rem} \label{rem:nzd} If $Z$ is non-degenerate and if $P_i=[a_{i,0},\ldots,a_{i,n}],i=1,\ldots,m$, the existence of $\tilde{\ell}\in R_1$ such that $\tilde{\ell}(P_i)=\frac{d_i}{c_i}$ for all $i=1,\ldots,m$ is equivalent to the $m\times (n+2)$ matrix
$$
\left(
\begin{array}{cccc}
a_{1,0}  & \cdots & a_{1,n} &-d_1/c_1\\
a_{2,0}  & \cdots & a_{2,n} &-d_2/c_2\\
\vdots & & \vdots&\vdots\\
a_{m,0}  & \cdots & a_{m,n} &-d_m/c_m
\end{array}
\right)
$$ having rank $n+1$.
\end{rem}

\begin{rem} \label{rem:sym} Denote the $\mathbb K-$vector space spanned by the partial derivatives of order $r-j$ of $F$ with ${\rm D}^{r-j}(F)$. By the shape of $F$, it is obvious that ${\rm D}^{r-j}(F)$ is a subspace of ${\rm Span}_{\mathbb K}\langle L_1^j,\ldots,L_m^j\rangle$.

If $$\dim_{\mathbb K} {\rm D}^{r-j}(F)= \dim_{\mathbb K} {\rm Span}_{\mathbb K}\langle L_1^j,\ldots,L_m^j\rangle - \dim_{\mathbb K} {\rm Span}_{\mathbb K}\langle L_1^{j-1},\ldots,L_m^{j-1}\rangle,$$ for all $j=0,\ldots r$, then $Z$ has symmetric $h-$vector.
\end{rem}

\medskip

\begin{exm} Consider $F=y_2^2+(y_0+y_1+y_2)^2-(y_0+y_2)^2-(y_1+y_2)^2$. Then $$Z=\{[0,0,1],[1,1,1],[1,0,1],[0,1,1]\}\subset\mathbb P^2.$$ Observe that $$\underbrace{y_2}_{L_1}+ \underbrace{(y_0+y_1+y_2)}_{L_2}-\underbrace{(y_0+y_2)}_{L_3}-\underbrace{(y_1+y_2)}_{L_4}=0$$ and that any three of the linear forms $L_1,L_2,L_3,L_4$ are linearly independent.

The rank of $
\left(
\begin{array}{cccc}
0  & 0 & 1 &-1\\
1  & 1 & 1 &-1\\
1 & 0 & 1&-1\\
0 & 1 & 1 &-1
\end{array}
\right)
$ is precisely $3$. The kernel of the matrix consists of vectors $\left(\begin{array}{cccc} 0 & 0 & a & a \end{array}\right)$, giving $\ell = x_2$.

So $\langle I_Z, x_2\rangle=Ann(F)=\langle x_0^2,x_1^2,x_2\rangle$, which is in line with the fact that $I_Z=\langle x_0(x_0-x_2),x_1(x_1-x_2)\rangle$.
\end{exm}

\vskip .2in

\textbf{Acknowledgements.} I would like to thank Anthony Iarrobino and Hal Schenck for the reference suggestions. Also I would like to thank Juan Migliore for the discussions we had on the subject, back in 2010. I am very grateful to the anonymous referee for important corrections and suggestions.

\renewcommand{\baselinestretch}{1.0}
\small\normalsize 

\bibliographystyle{amsalpha}

\end{document}